\documentclass[letterpaper, 10 pt, conference]{ieeeconf}

\IEEEoverridecommandlockouts                              

\usepackage[linesnumbered,ruled]{algorithm2e}
\usepackage{amsmath} 
\usepackage{amssymb}  
\usepackage{bm}
\usepackage{cleveref}
\usepackage{xcolor}
\usepackage{comment}
\usepackage[font=scriptsize]{caption}
\usepackage{graphicx} 
\usepackage[font=scriptsize]{subfig}
\usepackage{multicol}
\usepackage{siunitx}

\newtheorem{theorem}{Theorem}

\newtheorem{lemma}{Lemma}

\newtheorem{definition}{Definition}
\newtheorem{assumption}{Assumption}
\newtheorem{corollary}{Corollary}

\usepackage[maxbibnames=10,giveninits=true,sorting=none]{biblatex} 
\AtEveryBibitem{%
  \clearfield{issn} 
  \clearfield{doi} 
  \clearfield{url}
  \clearfield{pages}%
   \clearfield{number}
}
\DeclareFieldFormat*{year}{{}}
\renewbibmacro{in:}{} 

\usepackage{xcolor}

\title{An Optimal Solution to Infinite Horizon Nonlinear Control Problems}

\author{Mohamed Naveed Gul Mohamed, Raman Goyal, Suman Chakravorty 
\thanks{The authors are with the Department of Aerospace Engineering, Texas A\&M University, College Station, TX 77843, USA. \{\tt naveed, ramaniitrgoyal92, schakrav\}@tamu.edu
}}

\begin{document}

\maketitle

\begin{abstract}
In this paper, we consider the infinite horizon optimal control problem for nonlinear systems. Under the conditions of controllability of the linearized system around the origin, and nonlinear controllability of the system to a terminal set containing the origin, we establish an approximate regularized solution approach consisting of a ``finite free final time" optimal transfer problem to the terminal set, and an infinite horizon linear regulation problem within the terminal set, that is shown to render the origin globally asymptotically stable. Further, we show that the approximations converge to the true optimal cost function as the size of the terminal set decreases to zero. The approach is empirically evaluated on the pendulum and cart-pole swing-up problems to show that the finite time transfer is far shorter than the effective horizon required to solve the infinite horizon problem without the proposed regularization.
\end{abstract}
\begin{keywords}
Nonlinear control, Infinite horizon optimal control, Control Lyapunov function
\end{keywords}
\section{Introduction}
The goal of an optimal control problem is to find the control inputs that minimize a given cost function subject to constraints on the system dynamics. Further, it is desired that the optimal control problem results in a globally asymptotically stable closed-loop system.  In order to satisfy the global asymptotic stability requirement, one has to pose an infinite horizon control problem, i.e., where the cost is an infinite sum.
Although there exists infinite horizon optimal control law for specific, mostly linear, systems, obtaining an optimal solution for an infinite horizon nonlinear control problem is a very challenging task owing to the infinite horizon \cite{bertsekas_vol1}. In this work, we propose an approximate solution to the infinite horizon optimal control problem by turning the problem into an equivalent finite horizon problem. In particular, we note that if the infinite horizon cost is finite, the tail sum of the cost has to vanish, which implies that the system has to spend most of the infinite time around the origin. Therefore, assuming that the system linearization is controllable, we can approximate the tail cost after some finite time with the optimal linear cost function obtained by solving the stationary Riccati equation. Thus, the infinite horizon problem reduces to finding the optimal insertion time, along with the associated control inputs, into a level set of the optimal linear cost function that contains the origin. As the size of these level sets is decreased, the approximations are shown to converge to the true optimal cost function. Furthermore, although the solution may not be optimal for a finite level set, it nonetheless provides a control Lyapunov function that can globally asymptotically stabilize the underlying system, under a mild nonlinear controllability assumption that any state can be controlled into a terminal set containing the origin. \\

The infinite horizon optimal control problem may equivalently be posed as the stationary Dynamic Programming (DP) problem for discrete-time systems, or the stationary Hamilton-Jacobi-Bellman (HJB) equation in continuous time systems \cite{bertsekas_vol1,Bellman:1957} which turns the sequential decision making problem into a single step decision making problem given one knows the optimal cost function. It is also well known that the optimal solution to the stationary HJB equation is a globally asymptotically stabilizing control Lyapunov function \cite{bernstein1993nonquadratic,wan1992family, wan1995nonlinear}. However, it is also well known that the solution is computationally intractable owing to Bellman's dreaded ``curse of dimensionality" \cite{bertsekas_vol1,Bellman:1957}. Thus, there is a very large literature on Approximate DP (ADP) and Reinforcement Learning  that seeks to alleviate the curse of dimensionality. Approximate dynamic programming methods \cite{ADP_handbook,lewis2013reinforcement} typically use aggregation methods \cite{bertsekas1989adaptive}  \cite{haskell2016empirical} or various function approximation methods \cite{geramifard2013tutorial} \cite{munos2008finite} to give an approximately optimal policy/value function with a high confidence.  Reinforcement learning  \cite{sutton2018reinforcement} and approximate dynamic programming are typically used in the same sense, though the keyword reinforcement learning is often associated with a model-free approach where one does not have ``explicit" knowledge of the model, and instead, seeks to improve the control policy  by repeated interactions with the environment while observing the system's responses. The repeated interactions, or learning trials, allow these algorithms to compute the solution of the dynamic programming problem (optimal value/Q-value function or optimal policy) without explicitly constructing the model of the unknown dynamical system. Standard RL algorithms are broadly  divided into value-based methods, like Q-learning \cite{watkins1992q}, and policy-based methods, like policy gradient algorithms \cite{sutton2000policy}. Recently, function approximation using deep neural networks has significantly improved the performance of reinforcement learning algorithms, leading to a growing class of literature on `deep reinforcement learning' \cite{silver2016mastering1, schulman2015trust,SAC,TD3}. Despite the success on relatively higher dimensional problems than previously possible, the inherent variance in the solution \cite{henderson2018deep,D2C2.0_CDC,arxiv_D2C2.0} renders them unreliable and the training time required of these methods remain prohibitive.  \\

An alternative ``direct" approach to the HJB is to solve the underlying infinite horizon optimal control problem given a particular initial state. The field of Model Predictive Control (MPC) takes this approach to solving the infinite horizon problem, however, owing to the infinite horizon nature of the involved optimal control problem, MPC solves a ``fixed final time" finite horizon problem in its stead, takes the first control action, and repeats the process once at the next state \cite{mayne2014model,mayne2000constrained}. The ``stabilizing ingredient" for the asymptotic stability of the MPC approach is the use of a suitable terminal cost function in the optimization problem that is a control Lyapunov function for the system in some terminal set containing the origin \cite{mayne2014model}.  Nonetheless, the domain of attraction of the MPC law can be undesirably small, and thus, different methods have been suggested to increase the domain of attraction \cite{MPC_GAS1,MPC_GAS2,MPC_GAS3}. Alternatively, one can eschew the use of a terminal cost function and set using a suitable long horizon \cite{grune2017}, but this typically leads to intractability owing to very long prediction horizons \cite{MPC_FTC1}.
Our approach to solving the infinite horizon problem is similar to MPC in that we ``directly" solve the optimal control problem but the key difference is that given a state, we solve a ``free-final time" problem for insertion into a terminal set in which the optimal linear feedback law is asymptotically stabilizing and/or a good approximation of the true optimal cost function. We can then establish the global asymptotic stability of the resulting feedback law under a mild nonlinear controllability assumption, and also show that the approximation converges to the true optimal cost function as the size of the terminal set is reduced to zero. Finally, there is no need for replanning in the approach. The limitation is that we do not consider state or control constraints in the problem. \\

The primary contribution of this paper is a tractable direct approach for the solution of infinite horizon optimal control problems that is globally asymptotically stabilizing for nonlinear systems under the conditions of controllability of the system linearization around the origin, and a nonlinear controllability assumption into a terminal set containing the origin. The rest of the paper is organized as follows: we introduce the problem in Section II, the solution approach is detailed in Section III including the heuristic idea behind the construction, and the method is tested on empirically on several nonlinear systems in Section IV.

 

\section{Preliminaries}\label{section:prob}
Consider the following problem which we want to solve:
\begin{subequations}\label{eq.IHOCP}
\begin{align}
    J_{\infty}^*(x) = \min_{\{u_t\}} \sum_{t=0}^{\infty} c(x_t, u_t);~~ &\text{given} ~x_0 = x \label{eq.infinite_horizon_cost_func} \\
    \text{subject to the dynamics: } x_{t+1} &= f(x_t, u_t), \label{eq.dynamics}
\end{align}
\end{subequations}
where $x_t \in \mathbb{R}^n$ represents the state of the dynamical system, $u_t \in \mathbb{R}^p$ represents the control input to the dynamical system, and $c(x_t, u_t)$ is the incremental cost incurred in taking control action $u_t$ at state $x_t$. The above problem is an infinite horizon optimal control problem (IH-OCP), and thus, solving the problem is, in general, intractable owing to the infinite horizon of the problem. 

Our goal in this work is to develop a tractable approach to solving the above problem by transforming the problem into a suitable finite horizon problem.

Given that we can obtain a solution to the IH-OCP \cref{eq.IHOCP}, it is well known that the infinite horizon cost-to-go $J_{\infty}(\cdot)$ satisfies Bellman's equation \cite[Ch.7]{bertsekas_vol1}:
\begin{align}
    J_{\infty}^*(x) = \min_{u} \{ c(x, u) + J_{\infty}^*(f(x, u)) \}. \label{eq.bellman}
\end{align}

Given that $c(x,u) > 0 ~\forall ~(x,u)\neq (0,0)$, one may see that $J^*_{\infty}(\cdot)$ is a Control Lyapunov Function (CLF) for the dynamical system in \cref{eq.dynamics}, and thus, the control feedback policy implicitly defined by the optimal cost-to-go function $J^*_{\infty}(\cdot)$, globally asymptotically stabilizes the dynamical system \cref{eq.dynamics}.  The proof is straightforward and shown below.
\begin{corollary} \label{corollary.2}
    Let $J^*_\infty(x)$ satisfy the Bellman equation \cref{eq.bellman}, then it is a control Lyapunov function for the system in \cref{eq.dynamics} that renders the origin globally asymptotically stable.
\end{corollary}
\begin{proof}
    Let $u^*_t$ be the control action at state $x_t$ according to the Bellman equation, where $(x_t,u_t^*)$ is the system trajectory given $x_0 = x$. Then: 
    \begin{align*}
        J^*_\infty(x_t) &= c(x_t,u_t^*) + J^*_\infty(f(x_t,u^*_t)).\\
        \Rightarrow J^*_\infty(x_t) &> J^*_\infty(x_{t+1}) ~~(\text{since} ~c(x_t, u_t^*) >0).\\
        \Rightarrow J^*_\infty(x_t)&\to 0 ~\text{as} ~ t\to \infty ~~,
   \end{align*}
     where the last line follows from the fact that $J_{\infty}^* (x_t)$ is bounded below by zero, and thus, strict monotonicity of the sequence $\{J^*_{\infty}(x_t)\}$ implies that the sequence has to converge to zero. This holds for any choice of $x$. Since $J_{\infty}^*(x') = 0$ only at the origin, this implies the trajectory $x_t \rightarrow 0$ as $t\rightarrow \infty$ and the result follows.
\end{proof}

Further, suppose that if there exists a $J_{\infty}(\cdot)$ such that it satisfies the Bellman eq. (not necessarily optimal)
\begin{align}
    J_{\infty}(x) = \min_{u} \{ c(x, u) + J_{\infty}(f(x, u)) \}, \label{eq.suboptimal_bellman}
\end{align}
then $J_{\infty}(\cdot)$ also is a CLF that renders the origin globally asymptotically stable (GAS).

Thus, another goal for us in solving \cref{eq.IHOCP} is to construct CLFs as in \cref{eq.bellman}/ \cref{eq.suboptimal_bellman}, such that they render the origin GAS.

\section{Solution to the Infinite Horizon Optimal Control Problem}\label{section:sol_IHOCP}
Define the following finite-horizon optimal control problem (FH-OCP):
\begin{align}\label{eq.FHOCP}
    J^T_\infty(x) &= \min_{\{u_t\}}  \sum_{t=0}^{T-1} c(x_t,u_t) +\bar{J}_\infty(x_T),\\
    \text{subject to:} &~ x_{t+1} = f(x_t, u_t) \nonumber
\end{align}
where $\bar{J}_\infty(\cdot)$ is a terminal cost function that shall be defined below. 

We shall make the following assumptions for the rest of the paper. 
\begin{assumption}\label{assump.1 cost}
We assume that the cost function $c(x,u)$ has a global minimum at $(x,u) = (0,0)$, i.e., $\frac{\partial c}{\partial x} \Bigr|_{x=0, u=0} = 0$ and $\frac{\partial c}{\partial u}\Bigr|_{x=0, u=0} = 0$, $c(0,0) = 0$, and $c(x,u) > 0$ $\forall ~(x,u)\neq(0,0)$.
\end{assumption}
\begin{assumption}\label{assump.2 controllability}
    We assume that given any $x\in \mathcal{X}$, and any $\Omega \subset \mathcal{X}$, such that the origin is in $\Omega$, $\exists$ a control sequence $\{u_t\}_{t=0}^{T(x)}$,  that ensures $x_{T(x)} \in \Omega$ for some $T(x) < \infty$, under the dynamics defined above (\cref{eq.dynamics}). 
\end{assumption}
Assumption~\ref{assump.2 controllability} is a controllability assumption that ensures that any state can be controlled into entering the region $\Omega$ in finite time. 
\begin{assumption}\label{assump.3 linear controllability}
    We assume that the linearization of the dynamical system \cref{eq.dynamics} around $(0,0)$, is controllable.\\
\end{assumption}

Given assumption~\ref{assump.1 cost} and \ref{assump.3 linear controllability}, we can define the optimal ``linear" infinite-horizon problem:
\begin{subequations}\label{eq.terminal_lqr}
\begin{align}
    \bar{J}_\infty(x) &= \min_{\{u_t \}} \sum_{t=0}^{\infty} (x_t' Q x_t + u_t' R u_t), \label{eq.quad_cost} \\
    \text{subject to:} ~& x_{t+1} = A x_t + B u_t, 
\end{align}
\end{subequations}
where, $(Q,R)$ and $(A,B)$ are obtained by performing a quadratic expansion of $c(x,u)$, and a linear expansion of the dynamics in \cref{eq.dynamics} around the origin $(x,u) = (0,0)$.

Note that owing to the linear controllability assumption \ref{assump.3 linear controllability}, $\bar{J}_\infty(\cdot)$ above may be found by solving the stationary algebraic Riccati equation, resulting in 
\begin{align}
     \bar{J}_\infty(x) = x' P_{\infty} x,
\end{align}
where $ P_{\infty} $ is the solution of the stationary Riccati equation.
\begin{definition}
The terminal cost $\bar{J}_\infty(\cdot)$ in the finite horizon optimal control problem (\cref{eq.FHOCP}) above is defined as $\bar{J}_\infty(x) = x' P_{\infty} x,$ where $P_{\infty}$ is the stationary Riccati equation's solution.
\end{definition}
\subsection{Heuristic Idea}
The heuristic idea behind the solution of the infinite horizon optimal control problem is as follows. Let the optimal trajectory be given by $(x_t^*,u_t^*)$ starting at some state $x$. For the optimal cost to be well defined, we need that $\sum_{t=0}^{\infty} c(x_t^*,u_t^*) < \infty$. However, for this infinite sum to be well defined, the tail sum $\sum_{t>T}^{\infty} c(x_t^*,u_t^*) \rightarrow 0$ as $T \rightarrow \infty$. Thus, after some finite time, the cost $c(x_t^*,u_t^*)$ is necessarily small. Given that due to assumption \ref{assump.1 cost} the cost $c(x,u) = 0$ only at the origin, it follows that the system spends a very large time around the origin. Next, given assumption \ref{assump.3 linear controllability} that the linearization around the origin is controllable, it follows that the tail sum of the cost can be well approximated by the optimal linear cost function found by solving the stationary Riccati equation, $\bar{J}_{\infty}$, as defined above. Thus, the infinite horizon cost may be split into two parts: a ``transfer cost" to get the system into the region $\Omega$ where the linear approximation holds, and a linear regulation cost for regulation to the origin  once within $\Omega$ (see Fig. \ref{inf_horizon}).  Thus, the basic idea is to turn the infinite horizon problem into a finite horizon problem by approximating the tail sum of the cost by the optimal linear cost function and find the optimal insertion time $T^*(x)$ from state $x$ into the region $\Omega$ along with the associated optimal control. As the region $\Omega$ gets smaller, we get better approximations of the optimal cost function and obtain the optimum in the limit. However, even for a finite $\Omega$, this procedure results in the construction of a CLF that renders the origin globally asymptotically stable. In fact, we can get an uncountable number of such CLFs based on the choice of our incremental cost function $c(x,u)$.

\begin{figure}[t]
    \centering
    \includegraphics[width = \linewidth]{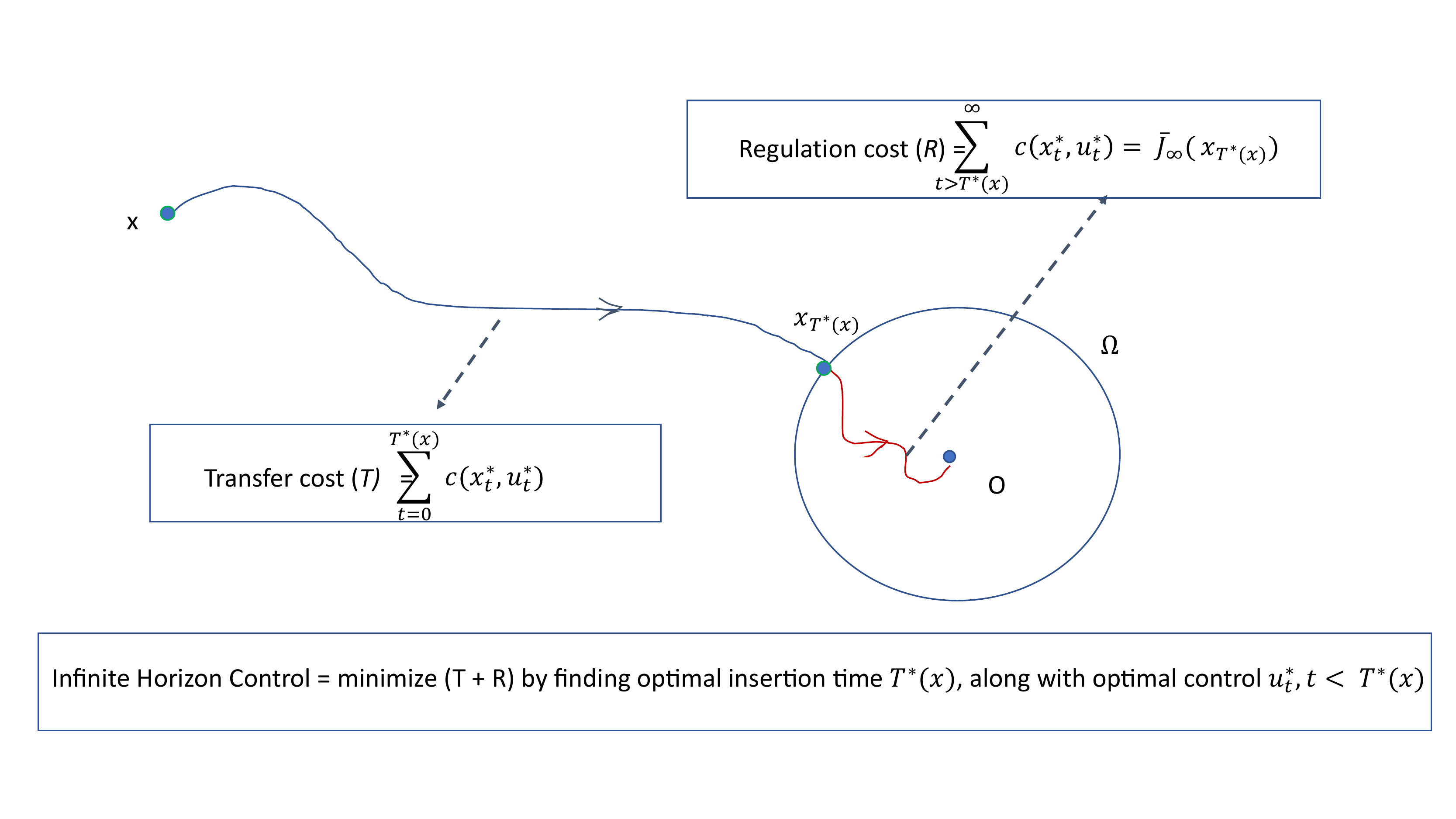}
    \caption{Schematic illustrating the strategy to solve the infinite horizon optimal control problem}
    \label{inf_horizon}
\end{figure}

\subsection{Optimality of the Finite Horizon Optimal Control Problem}
In the following, we shall establish the optimality of the FH-OCP \cref{eq.FHOCP}, in particular, that $J^T_\infty(x) \rightarrow {J}^*_\infty(x)$ for all $x\in \mathcal{X}$.

To this end, define the set:

    $\Omega_\varepsilon = \{x: |\bar{J}_\infty(x)- {J}^*_\infty(x)| < \varepsilon,$ and   $\bar{J}_\infty(x) = x' P_{\infty} x \leq M_\varepsilon\}$, i.e., $M_{\epsilon}$ is chosen such that in the level set $\Omega_{\varepsilon}$, the error between the linear optimal cost function and the true optimal cost function is less than the tolerance $\epsilon$.

The set above defines a sub-level set of the optimal linear cost function in which the cost function is $\varepsilon-$close to the true (unknown) optimum $J^*_{\infty}(x)$. 

For simplicity, we will consider the cost $c(x,u)$ to be quadratic in the following.
\begin{lemma}\label{lemma.1}
    There exists a finite time $T_{\varepsilon} < \infty$, such that the solution to the FH-OCP \cref{eq.FHOCP} for $T = T_\varepsilon$, denoted by $(x_t^\varepsilon, u_t^\varepsilon)$ is such that $\bar{J}_{\infty}(x_T) = x^{\varepsilon '}_T P_\infty x^\varepsilon_T \leq M_\varepsilon$ for the first time, i.e., for any $T<T_\varepsilon$,  $\bar{J}_{\infty}(x_T) = x^{\varepsilon '}_T P_\infty x^\varepsilon_T > M_\varepsilon$.
\end{lemma}
\begin{proof}
    If $T_\varepsilon$ is finite, then by definition, $\bar{J}_\infty (x^\varepsilon_{T_\varepsilon}) \leq M_\varepsilon$ for the first time at $T_\varepsilon$. Owing to assumption \ref{assump.2 controllability}, there exits a control sequence $\{u_t \}_{t=0}^{\bar{T}}$ such that $\bar{x}_{\bar{T}} \in \Omega_\varepsilon$ after some finite time $\bar{T}$. Denote this cost by $\bar{J}$. Since the solution for \cref{eq.FHOCP} does not enter $\Omega_\varepsilon$ for any $T$, and since the cost $c(x,u) \geq \Delta > 0$, if $x \notin \Omega_\varepsilon$ (owing to assumption \ref{assump.1 cost}), $J^T_\infty \rightarrow \infty$ as $ T \rightarrow \infty$. However, owing to the trajectory $(\bar{x}_t, \bar{u}_t)$ from above: $\bar{J}(x) = \sum_{t=0}^{\bar{T}-1} c(\bar{x}_t, \bar{u}_t) + \sum_{t \geq \bar{T}} c(\bar{x}_t, \bar{u}_t),$ where $\bar{T}$ is the time the trajectory hits $\Omega_\varepsilon$ with $\bar{u}_t$ following the linear control after $\bar{T}$. Choose a $T$ such that $\forall ~T > T'$, $J^T(x) > \bar{J}(x) + \varepsilon$ (this is always possible since $J^T(x) \rightarrow \infty$). Now, choose a $T$ such that $T > \bar{T}$ and $T > T'$. This implies: $J^T_{\infty}(x) = \sum_{t=0}^{T-1} c(x_t, u_t) + \bar{J}_\infty(x_T) > \sum_{t=0}^{T'-1} c(\bar{x}_t, \bar{u}_t) + \sum_{t \geq T'} c(\bar{x}_t, \bar{u}_t) + \varepsilon$.
    But  $\sum_{t \geq T} c(\bar{x}_t, \bar{u}_t) > J^*_\infty(\bar{x}_T) \geq \bar{J}_\infty(\bar{x}_T) - \varepsilon.$\\
    $\Rightarrow J^T_\infty(x) = \sum_{t=0}^{T-1} c(x_t, u_t) + \bar{J}_\infty (x_T) > \sum_{t=0}^{T-1} c(\bar{x}_t, \bar{u}_t) + \bar{J}_\infty (\bar{x}_T),$ which contradicts the fact that $J^T_\infty(x)$ is the optimal cost for the time $T$. Thus, $\exists$ a finite time $T_\varepsilon < \infty$ such that the trajectory under the solution to \cref{eq.FHOCP} hits the set $\Omega_\varepsilon.$
\end{proof}

The basic idea in the proof above is that if the optimal trajectory never enters $\Omega_\varepsilon, ~J^T \rightarrow \infty$ starting from $x$, however, since there is a feasible trajectory with $\bar{J} < \infty,$ we can always find a suitable finite time $T$ where the cost of the optimal control is above $\bar{J}$ thereby contradicting optimality at that time $T$.\\

\begin{lemma} \label{lemma.2}
    Given the time $T_\varepsilon$ that the solution to \cref {eq.FHOCP} first hits the set $\Omega_\varepsilon$, the optimal solution $T_\varepsilon$ satisfies: $| J^{T_\varepsilon}_\infty(x) - J^*_\infty(x)| \leq \varepsilon$.
\end{lemma}
\begin{proof}
    Let $(x^*_t, u^*_t)$ denote the optimal trajectory and let $T^*_\varepsilon$ be the time it hits $\Omega_\varepsilon$. There are three possibilities: $T^*_\varepsilon > T_\varepsilon, ~T^*_\varepsilon < T_\varepsilon, ~\text{and} ~T^*_\varepsilon = T_\varepsilon$.\\
    
    (a) $T^*_\varepsilon > T_\varepsilon$: Let $ J^*_\infty(x) = \sum_{t=0}^{T^*_\varepsilon - 1} c(x^*_t, u^*_t) + J^*_\infty(x^*_{T^*_\varepsilon})$.\\
    $J^{T_\varepsilon}_\infty(x) = \sum_{t=0}^{T_\varepsilon - 1} c(x^\varepsilon_t, u^\varepsilon_t) + \bar{J}_\infty(x^\varepsilon_{T_\varepsilon}),$ where $(x^\varepsilon_t, u^\varepsilon_t)$ is the optimal trajectory from solving \cref{eq.FHOCP} for time $T_\varepsilon$ and finally let $\bar{J}^*_\infty(x) = \sum_{t=0}^{T^*_\varepsilon - 1} c(x^*_t, u^*_t) + \bar{J}^*_\infty(x^*_{T^*_\varepsilon}).$ \\
    But owing to the defintion of $\Omega_\varepsilon$, $|\bar{J}^*_\infty(x) - J^*_\infty(x)| \leq \varepsilon.$ But $\bar{J}^*_\infty(x) = \sum_{t=0}^{T_\varepsilon -1} c(x^*_t, u^*_t) + \sum_{t>T_\varepsilon}^{T^*_\varepsilon} c(x^*_t, u^*_t) + \bar{J}_\infty(x^*_{T^*_\varepsilon}).$ Noting that $x^*_{T^*_\varepsilon}$ and $x^{\varepsilon}_{T_\varepsilon}$ are both on $\partial \Omega_\varepsilon$ (boundary of the set $\Omega_\varepsilon$) and noting that $J^{T_\varepsilon}_\infty(x)$ is optimum for $T_\varepsilon$ in \cref{eq.FHOCP}, it follows that $J^{T_\varepsilon}_\infty(x) < \bar{J}^*_\infty(x) \Rightarrow |J^{T_\varepsilon}_\infty(x) - \bar{J}^*_\infty(x)| < \varepsilon.$ \\

    (b) $T^*_\varepsilon < T_\varepsilon $: Let the solution to the FH-OCP for $T^*_\varepsilon$ be $(\tilde{x}_t, \tilde{u}_t)$, and define:
    \begin{align*}
        J^*_\infty(x) &= \sum_{t=0}^{T^*_\varepsilon - 1} c(x^*_t, u^*_t) + J^*_\infty(x^*_{T^*_\varepsilon}). \\
        J^{T_\varepsilon}_\infty(x) &= \sum_{t=0}^{T^*_\varepsilon -1} c(x^\varepsilon_t, u^\varepsilon_t) +\sum_{t \geq T^*_\varepsilon}^{T_\varepsilon - 1} c(x^\varepsilon_t, u^\varepsilon_t) + \bar{J}_\infty(x^\varepsilon_{T_\varepsilon}).\\
        J^{T^*_\varepsilon}_\infty (x) &= \sum_{t=0}^{T^*_\varepsilon - 1} c(\tilde{x}_t, \tilde{u}_t) + \bar{J}_\infty(\tilde{x}_{T^*_\varepsilon}).\\
        \bar{J}^*_\infty(x) &= \sum_{t=0}^{T^*_\varepsilon -1} c(x^*_t, u^*_t)   + \bar{J}_\infty(x^*_{T^*_\varepsilon})
    \end{align*}
Owing to the definition of $\Omega_\varepsilon$: 
\begin{align}
    |\bar{J}^*_\infty(x) - J^*_\infty(x)| < \varepsilon \label{eq.lemma2_proof i}
\end{align}
(note that $x^*_{T^*_\varepsilon} \in \partial \Omega_\varepsilon)$. But $J^{T^*_\varepsilon}_\infty(x)$ is the optimum to FH-OCP \cref{eq.FHOCP} for $T^*_\varepsilon \Rightarrow$
\begin{align}
    J^{T^*_\varepsilon}_\infty(x) < \bar{J}^*_\infty(x).
\end{align}
Finally: 
\begin{align}
    J^{T^*_\varepsilon}_\infty(x) = \sum_{t=0}^{T_\varepsilon- 1} c(\tilde{x}_t, \tilde{u}_t) + \bar{J}_\infty(\tilde{x}_{T_\varepsilon}),
\end{align}
where $\tilde{u}_t = u_l(\tilde{x}_t)~\forall~ t> T^*_\varepsilon,$ where $u_l(\cdot)$ denotes the linear control according to $\bar{J}_{\infty}(\cdot).$ Comparing this to $J^{T_\varepsilon}_\infty(x)$, and noting that $J^{T_\varepsilon}_\infty(x)$ is optimal for $T_\varepsilon \Rightarrow$
\begin{align}
    J^{T_\varepsilon}_\infty(x) \leq J^{T^*_\varepsilon}_\infty(x). \label{eq.lemma2_proof iv}
\end{align}
Putting \cref{eq.lemma2_proof i}-\cref{eq.lemma2_proof iv} together $\Rightarrow |J^*_\infty - J^{T_\varepsilon}_\infty(x)| \leq \varepsilon$. \\

(c) The case $T_\varepsilon = T^*_\varepsilon$ is trivial, and thus, the result is proved.
\end{proof}

\begin{corollary}\label{corollary.1}
Let $T'<T,$ then the solution to \cref{eq.FHOCP} for $T,~T'$ follow: $J^{T'}_\infty(x) \geq J^T_\infty(x).$  The proof follows from the last part of the proof to lemma \ref{lemma.2} above.     
\end{corollary}

\par
\begin{theorem}\label{theorem.1}
    Given any $x \in \mathcal{X}$, $\lim_{T\rightarrow \infty} J^T_\infty(x) \rightarrow J^*_\infty(x)$.
\end{theorem}    
\begin{proof}
We just showed in lemma \ref{lemma.2} that given any $\varepsilon > 0$, we may find $T_\varepsilon(x) < \infty$ such that: 
\begin{align}
    |J^{T_\varepsilon}_\infty(x) - J^*_{\infty}(x)| \leq \varepsilon.
\end{align}
Further, owing to corollary \ref{corollary.1}, $J^{T}_\infty(x) \leq J^{T_\varepsilon}_\infty(x)$ for any $T > T_\varepsilon$. Noting that $J^{T_\varepsilon}_\infty(x) > J^*_\infty(x),$  it follows that $|J^T_\infty(x) - J^*_\infty(x)|\leq\varepsilon ~\forall~ T > T_\varepsilon(x)$, this establishes that $J^T_\infty(x) \rightarrow J^*_\infty(x)$ for any $x\in \mathcal{X}$.
\end{proof}

\subsection{An Alternative Construction}\label{section:suboptimal}
In the following, we define a sub-optimal construction that nonetheless furnishes a Control Lyapunov Function (CLF) that renders the origin globally asymptotically stable for the dynamical system \cref{eq.dynamics}.

Let the region of the attraction (ROA) of the linear controller corresponding to the terminal cost function $\bar{J}_\infty$ be $\mathcal{D}$. Suppose that we define a set $\Omega \subset \mathcal{D}$ as follows: $$\Omega = \{x:\bar{J}_\infty(x) = x' P_\infty x \leq M\},$$ which is a level set of $\bar{J}_\infty(\cdot)$ that is constrained in the ROA $\mathcal{D}$.  

Fix $M< \infty$, and consider the problem:
\begin{align}
    {J}_\infty(x) &= \min_{u_t, T} \sum_{t=0}^{T-1} c(x_t,u_t) +  \max \left(\bar{J}_\infty(x_T), M \right), \; x \notin \Omega, \nonumber\\
    {J}_\infty(x) &= \bar{J}_{\infty}(x) = x'P_{\infty} x, \; x \in \Omega.\label{eq:subopt}
\end{align}

Please note that the above optimization problem has a free final time $T$ that needs to be optimized over in conjunction with the control action $u_t$. The rationale behind posing the above problem is that the optimal control transfers the system from any state $x$ to the boundary $\partial\Omega$ of the set $\Omega$, from where the linear controller takes over to regulate to the origin. Since $\Omega$ is within the region of attraction $\mathcal{D}$, the linear controller can always regulate to the origin regardless of the insertion point on the boundary.\\

In the next result, we show that the optimal time for the problem posed above is the first time that the solution to the problem posed above hits the boundary of the set $\Omega$.
\begin{lemma}\label{lemma:subopt}
The optimal time $T^*(\Omega)$ for the optimal control problem (\cref{eq:subopt}) is finite and is the first time $T$ that the solution to \cref{eq:subopt} hits the set $\Omega$.
\end{lemma}
\begin{proof}
Note that \cref{eq:subopt} is the exact same as \cref{eq.FHOCP} as long as $x_T \notin \Omega$. Thus, \cref{lemma.1} implies that the first hitting time of $\Omega$, say $T^*(\Omega)$, is finite. 
Furthermore, \cref{corollary.1} implies that $J^T_\infty(x) < J^{T(\Omega)}_\infty(x)$ for any $T > T(\Omega)$, where:
\begin{align}
    J^T_\infty(x) = \min_{u_t}  \sum_{t=0}^{T-1} c(x_t,u_t) +\bar{J}_\infty(x_T).
\end{align}

Let $T>T(\Omega)$ and let the optimal trajectory be $\{\tilde{x}_t, \tilde{u}_t\}$ for the time $T$ according to the OCP in \cref{eq:subopt}. Then:
\begin{align}
    {J}^T_\infty(x) = \sum_{t=0}^{T-1} c(\tilde{x}_t,\tilde{u}_t) +  \max \left(\bar{J}_\infty(\tilde{x}_T), M \right). \label{eq:subopt1}
\end{align}
Since $T>T(\Omega)$, $\bar{J}_{\infty} (\tilde{x}_T)  < M $, and thus, $\max(\bar{J}_{\infty} (\tilde{x}_T), M) = M$. Noting that $T(\Omega)$ is the first time the solution to \cref{eq:subopt} hits $\partial \Omega$, $J^T_\infty(x) > J^{T(\Omega)}_\infty(x)$, and the result follows.
\end{proof}
Next, we show that the sub-optimal cost function thus constructed satisfies the Bellman equation.

\begin{lemma}
    The optimal cost $J_\infty(x)$ corresponding to OCP \cref{eq:subopt} satisfies the Bellman equation for any $x \notin \Omega$ :
    $$J_\infty(x) = \min_u [c(x,u) + J_\infty(f(x,u))].$$
\end{lemma}
\begin{proof} Let us denote the terminal cost function in the optimal control problem \cref{eq:subopt} as $\Phi(x_T)=\max \left(\bar{J}_\infty(x_T), M \right)$. Then for any $x \notin \Omega$:
\begin{align}
    {J}_\infty(x) &= \min_{u_t, T} \left[\sum_{t=0}^{T-1} c(x_t,u_t) +  \Phi(x_T)\right],\\
    \nonumber & = \min_{u_0} \min_{\{u_t\}_{t=1}^{T-1}, T} \left[c(x,u_0) +  \sum_{t=1}^{T-1} c(x_t,u_t) +  \Phi(x_T) \right],\\
    \nonumber & = \min_{u_0}  \left[c(x,u_0) + \min_{\{u_t\}_{t=1}^{T-1}, T} \sum_{t=1}^{T-1} \left[ c(x_t,u_t) +  \Phi(x_T) \right]  \right]. 
\end{align}

Noting that $x_1 = f(x,u_0)$ and $$ J_\infty(f(x,u_0)) = \min_{\{u_t\}_{t=1}^T, T} \sum_{t=1}^{T-1} \left[ c(x_t,u_t) +  \Phi(x_T) \right],$$ it follows that:
$$J_\infty(x) = \min_u [c(x,u) + J_\infty(f(x,u))],$$ establishing the result.\\
\end{proof}

Note that the reason $J_\infty(\cdot)$ satisfies the stationary Bellman equation is due to the free final time $T$ in the problem formulation as otherwise the sum would be time varying. Then, owing to the fact that $J_\infty(\cdot)$ satisfies the Bellman equation, one has:

\begin{corollary}
The optimal cost function $J_\infty(\cdot)$ of the OCP \cref{eq:subopt} is a control Lyapunov function that renders the origin of system \cref{eq.dynamics} global asymptotically stable.
\end{corollary}
\begin{proof}
    It is straightforward that for any $x \notin \Omega$, since $J_{\infty} (.)$ satisfies the Bellman equation outside of $\Omega$, a trajectory following the optimal control has to hit $\Omega$. Now, since $\Omega$ is within the domain of attraction $\mathcal{D}$ of the linear controller, $\bar{J}_{\infty}$ is a CLF for the system and the state $x_t \rightarrow 0$, as $t\rightarrow \infty$. Thus, the optimal cost function $J_{\infty}(.)$ is a CLF for the system.
\end{proof}
We note here that the CLF $J_{\infty}(.)$ as defined is continuous but not necessarily smooth, in particular, there is a continuous but non-smooth transition at the boundary of $\Omega$ for any finite $M$.\\

Furthermore, one can also establish the following convergence result. Given $M < \infty$, define:
\begin{align}
    {J}^M_\infty(x) = \min_{u_t, T} \left[\sum_{t=0}^{T-1} c(x_t,u_t) + \max \left(\bar{J}_\infty(x_T), M \right) \right], \label{eq:JM}
\end{align}
where recall that $\Omega = \{x:\bar{J}_\infty(x) = x' P_\infty x \leq M\}$.

\begin{theorem}
Given the optimal control problem (\cref{eq:JM}), the solution satisfies: $\lim_{M \rightarrow 0} {J}^M_\infty(x) = {J}^*_\infty(x)$ for all $x\in \mathcal{X}$.
\end{theorem}
\begin{proof}
The proof is essentially identical to that of Lemma 2 and Theorem 1. 
\end{proof}

Finally, note that ${J}^M_\infty(\cdot) \neq {J}^*_\infty(\cdot)$ for any finite $M<\infty$. However, ${J}^M_\infty(\cdot)$ still satisfies Bellman's equation, and this is a non-trivial case where satisfying the BE is not sufficient for optimality.

\subsection{Discussion}
In the following, we discuss some of the salient points of the proposed regularization approach.\\

\textit{Computational efficiency:} The approach proposed converges to the optimal cost function as the time horizon $T$ becomes large, or equivalently, the size of the terminal set characterized by $M$ becomes small. A valid question is how is this different from posing a problem without any terminal cost function and letting the time become large, where in the limit, the optimal cost function of the finite horizon problem converges to the optimal infinite horizon cost? To see the advantage of the proposed approach, let the optimal transfer time be $T^*_M(x)$ from a state $x$ into the set $\Omega_M = \{\bar{J}_{\infty} \leq M\}$. Further, denote the effective regulation time required within the set $\Omega_M$ as $T^R_M(x)$. Then, it is easy to see that $T^*_M(x) << T^R_M(x),$ for any finite $M$. Therefore, the optimization problem for our approach has much shorter horizon when compared to a finite horizon problem without any terminal cost. This aspect is clearly borne out by our experiments in the next Section where the effective horizon of the problem we need to solve is a small fraction of that required for the finite horizon problem without any terminal cost.\\

\textit{Free final time and global asymptotic stability:} Another critical aspect of our approach is the free final time of the optimization problem from a given state $x \in \mathcal{X}$. In fact, we have already noted that the free final time is the reason that the approximate cost function $J_{\infty}^M(.)$ satisfies the stationary Bellman equation, as opposed to a time varying Bellman equation, thereby establishing global aysmptotic stability. If we fixed the time horizon, then we would be back to the situation in MPC with the stabilizing ingredient of a terminal cost/ set constraint, and could only conclude asymptotic stability locally as in \cite{mayne2014model}. Physically, this makes sense since the optimal transfer time from different points in the state space will be different, and thus, the finite horizon in the optimal control problem needs to be variable. The free final time models precisely this aspect of the problem leading to the GAS closed loop system. The free final time problem can be solved efficiently by taking a ``large enough" horizon rather than sweping through time. Due to the argument above, ``large enough" in our case is much smaller than the ``large enough" time required in the problem without a terminal constraint. \\

\textit{Optimal control solution:} Our approach requires the solution of an optimal control problem, and in general, satisfying the Bellman equation requires that we find the global minimum of the nonlinear optimal control problem. In our recent work \cite{arxiv_PONA,PoNA_ACC}, we have shown that under the conditions of affine in control dynamics, and quadratic in control cost (which can be relaxed to a convex in control cost), satisfying the minimum principle \cite{bryson} is also sufficient to assure us of a global minimum. Albeit the problem is nominally non-convex, using the Method of Characteristics, we may show that the trajectories satisfying the minimum principle are guaranteed to be unique. Further, we have shown that the iterated LQR (ILQR) algorithm can be guaranteed to satisfy the minimum principle under relatively mild conditions, thereby assuring us that we can find the global minimum \cite{arxiv_D2C2.0,D2C2.0_CDC}. Hence, we use the ILQR approach to solve the optimal control problem in our experiments in Section IV.

\section{Empirical Results}\label{section:results}
The proposed theory is tested with simulations on two nonlinear systems: cart-pole and pendulum. The simulations were done on MATLAB, using ode45 to simulate the nonlinear dynamics of the two systems. The list of experiments performed is shown in Table \ref{tab:experiments}. The simulations are designed to take the system upright from different initial conditions and balance the pendulum/pole at the top. The simulations are run for a total of $150$ time-steps with a time discretization of $0.1s$. We take the cost function to be quadratic: $c(x,u) = x'Qx + u'Ru$. 

\begin{table}[!htbp]
    \centering
    \begin{tabular}{|c|c|c|c|}
        \hline
         Exp.\# & System & Initial state & Terminal state  \\
         \hline 
         1. & Cart-pole $(x,\theta, \dot{x}, \dot{\theta})$ & $(0,0,0,0)$ & $(0,\pi,0,0)$ \\
         \hline
         2. & Cart-pole $(x,\theta, \dot{x}, \dot{\theta})$ & $(0,3\pi/4,0,0)$ & $(0,\pi,0,0)$ \\
         \hline
         3. & Pendulum $(\theta, \dot{\theta})$ & $(0,0)$ & $(\pi,0)$ \\
         \hline 
         4. & Pendulum $(\theta, \dot{\theta})$ & $(5\pi/12,0)$ & $(\pi,0)$ \\
         \hline
    \end{tabular}
    \caption{Experiments performed}
    \label{tab:experiments}
\end{table}

\begin{figure}[!htbp]
\centering
    \sbox0{\subfloat[Total cost (Exp.1)]{\includegraphics[width=0.48\linewidth]{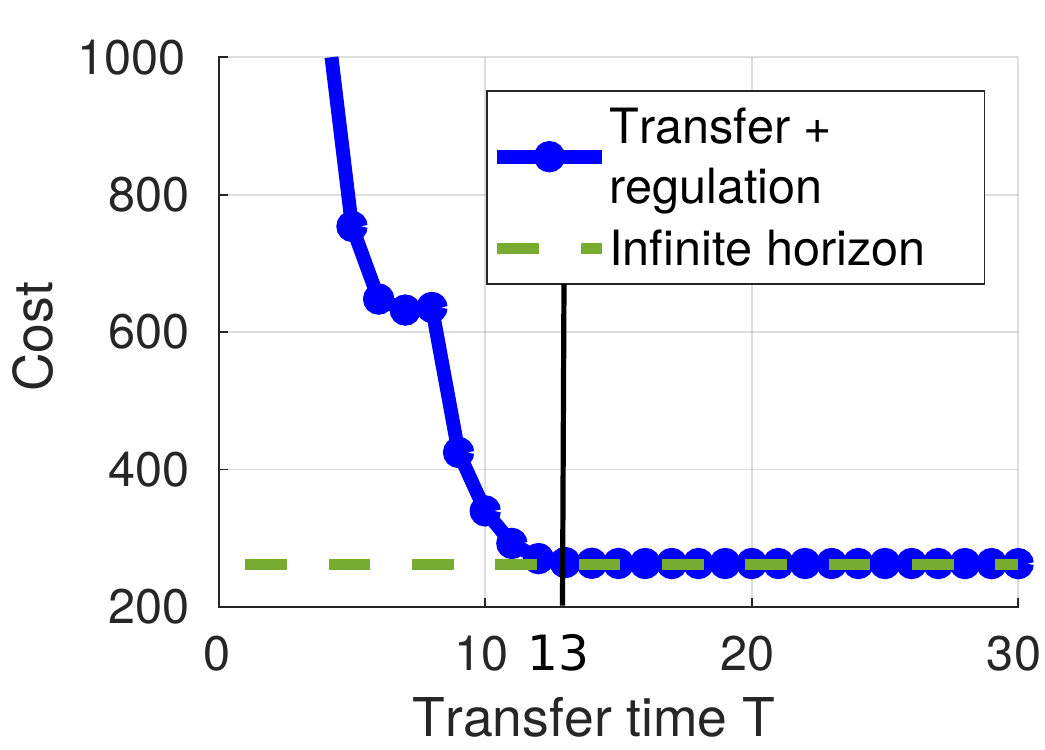}}}
    \sbox1{\subfloat[Terminal regulation cost (Exp.1)]{\includegraphics[width=0.48\linewidth]{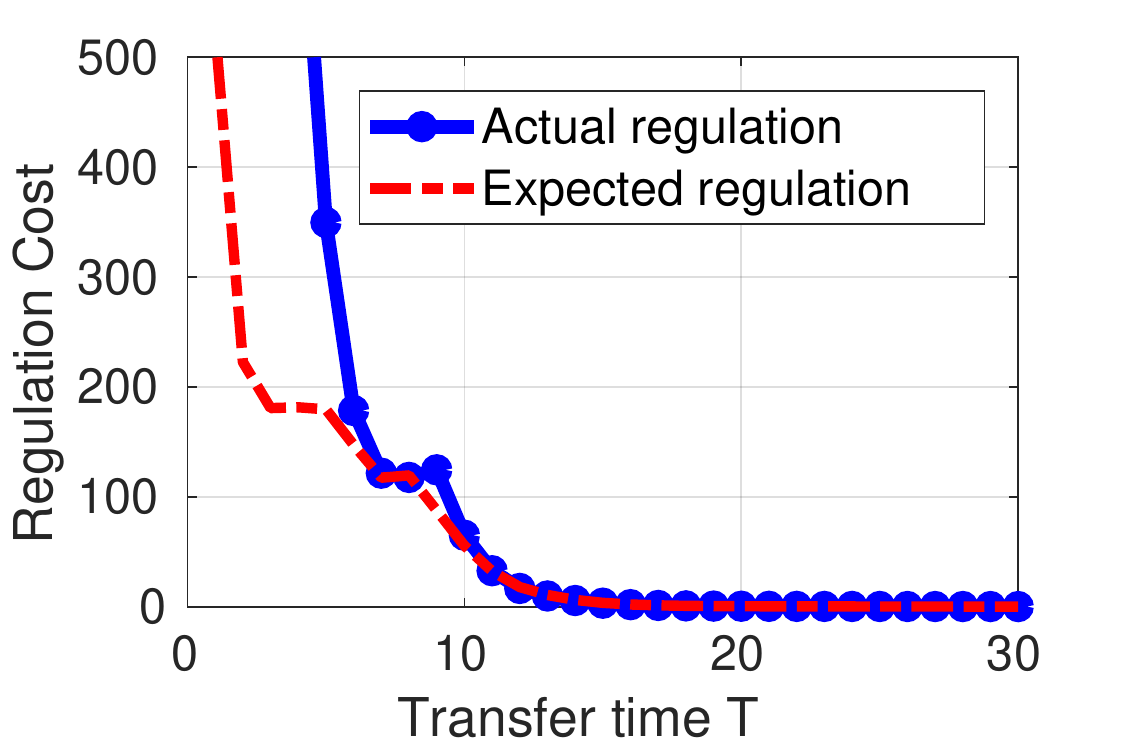}}}
    \sbox2{\subfloat[Error at the end of \\ finite horizon transfer (Exp.1)]{\includegraphics[width=0.48\linewidth]{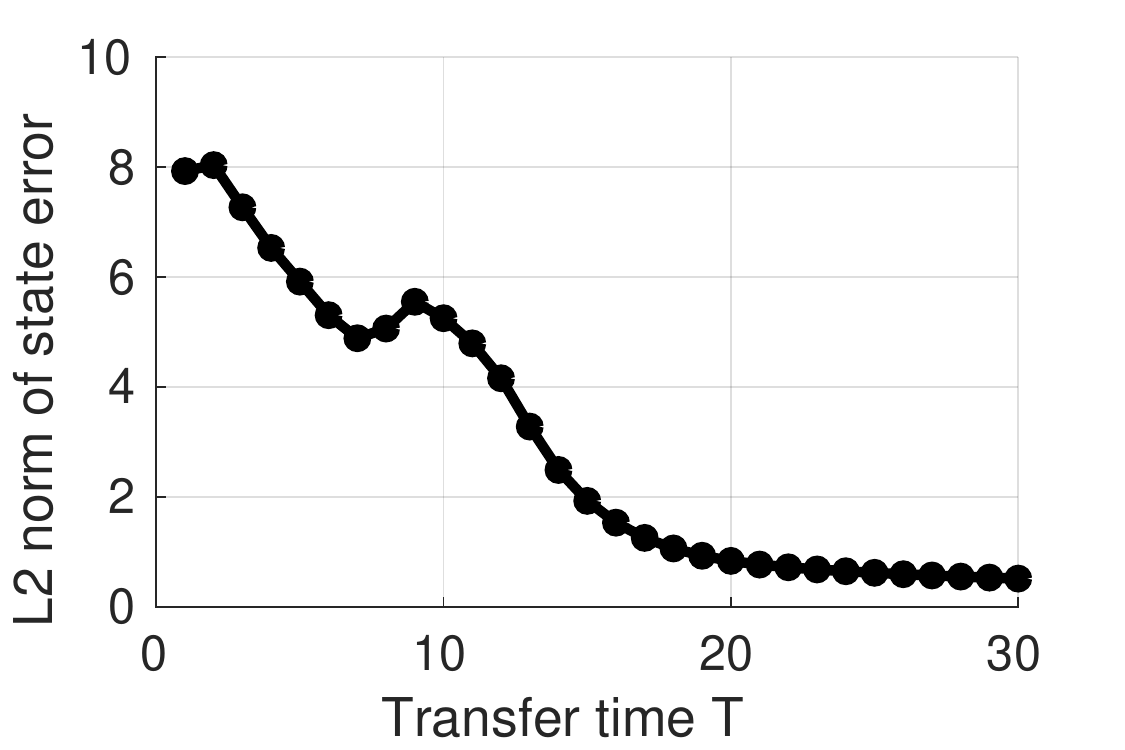}}}
    \sbox3{\subfloat[Total cost (Exp.2)]{\includegraphics[width=0.48\linewidth]{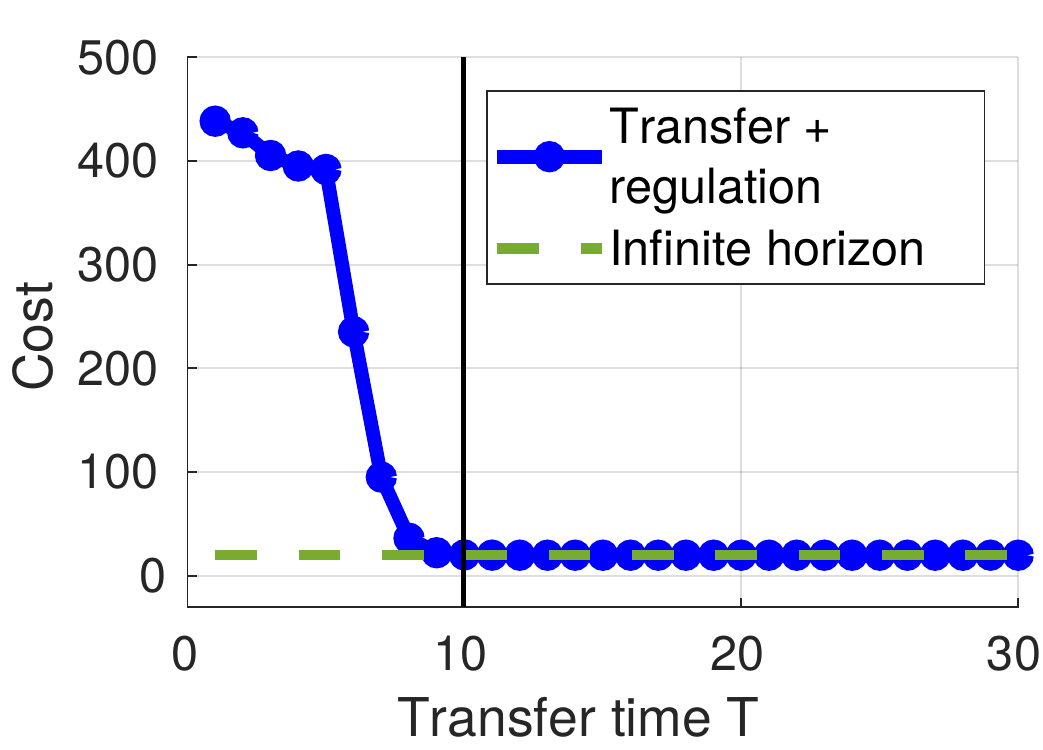}}}
    \sbox4{\subfloat[Terminal regulation cost (Exp.2)]{\includegraphics[width=0.48\linewidth]{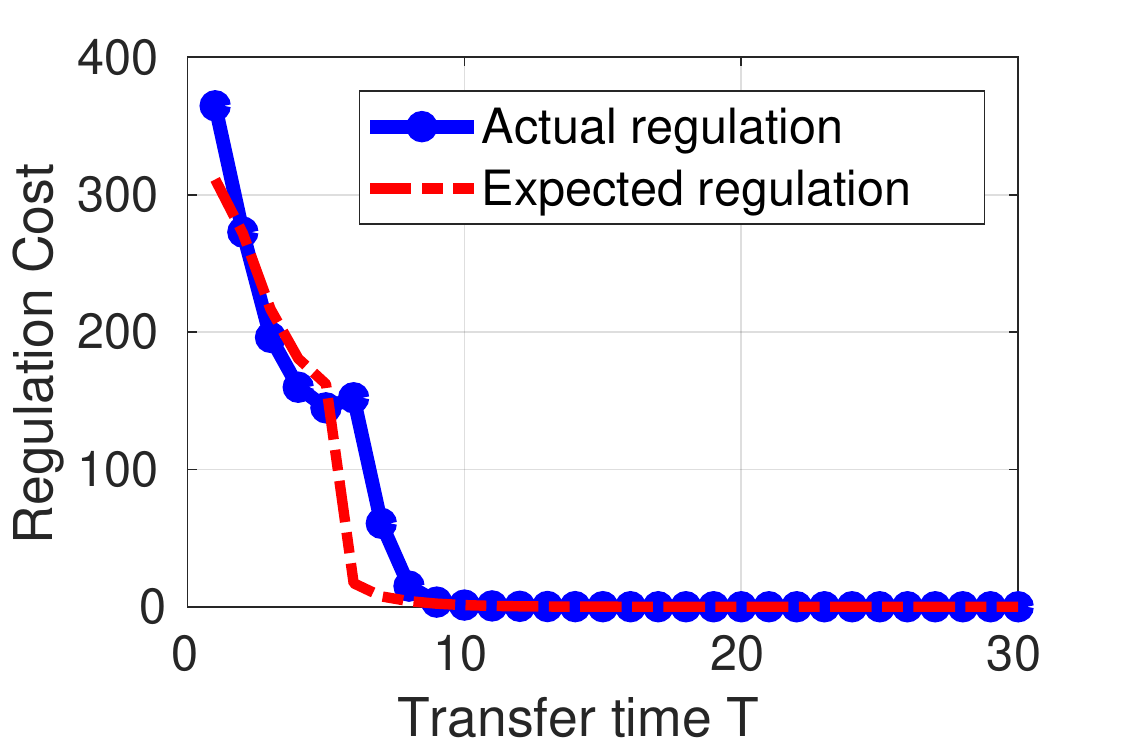}}}
    \sbox5{\subfloat[Error at the end of \\ finite horizon transfer (Exp.2)]{\includegraphics[width=0.48\linewidth]{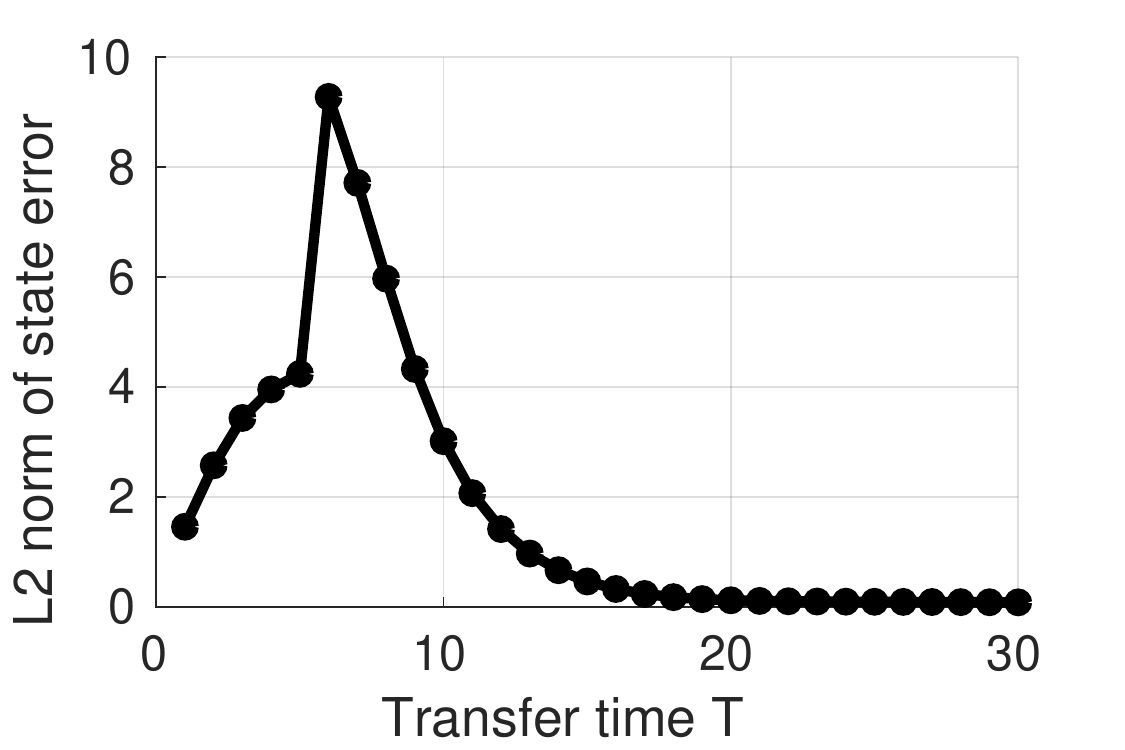}}}
    \centering
    \usebox0\hfil \usebox3\par
  \usebox1\hfil \usebox4\par
  \usebox2\hfil \usebox5
    \caption{Cart-pole  swing-up. Note plots labeled (a) - (c) are Exp.1 and (d) - (f) are Exp.2 in Table~\ref{tab:experiments}. The results in plots (a) and (d) show that the finite horizon transfer $+$ regulation costs converge to the infinite horizon cost after some time $T^*$ (marked with a black vertical line). One can also note that the value of $T^*$ in plots (a) and (d) are different since the system starts from different initial conditions.}
    \label{fig:cartpole}
\end{figure}

\begin{figure}[!htbp]
\centering
    \sbox0{\subfloat[Total cost (Exp.3)]{\includegraphics[width=0.48\linewidth]{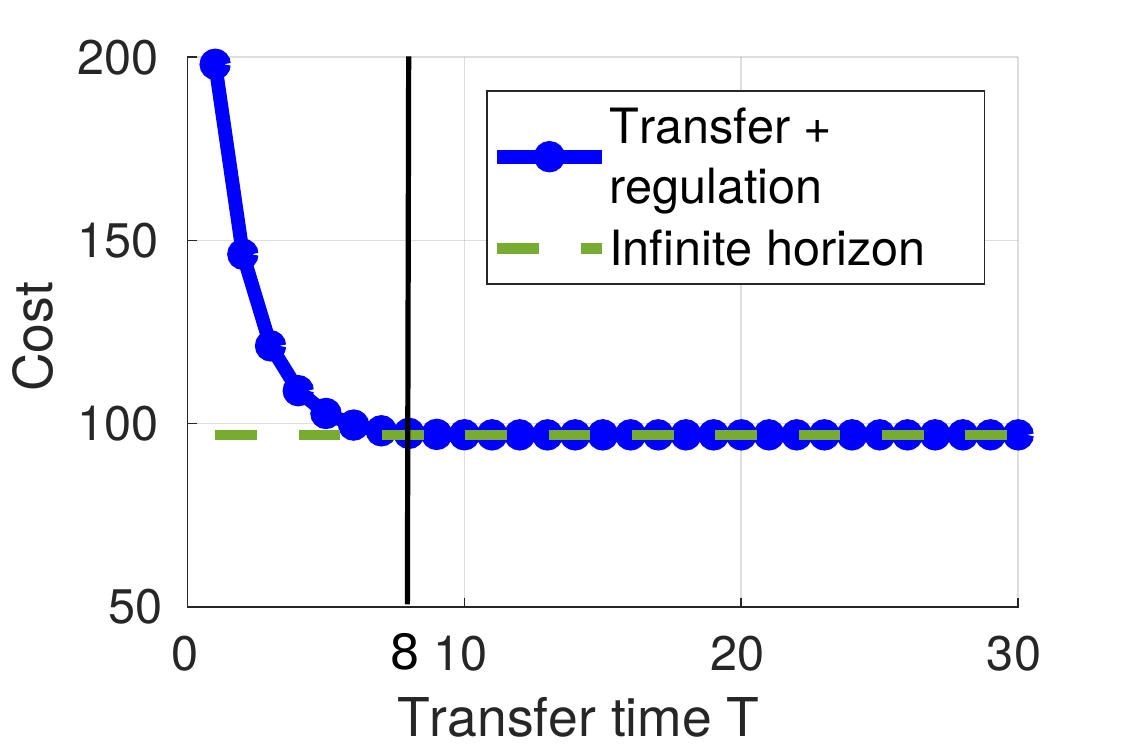}}}
    \sbox1{\subfloat[Terminal regulation cost (Exp.3)]{\includegraphics[width=0.48\linewidth]{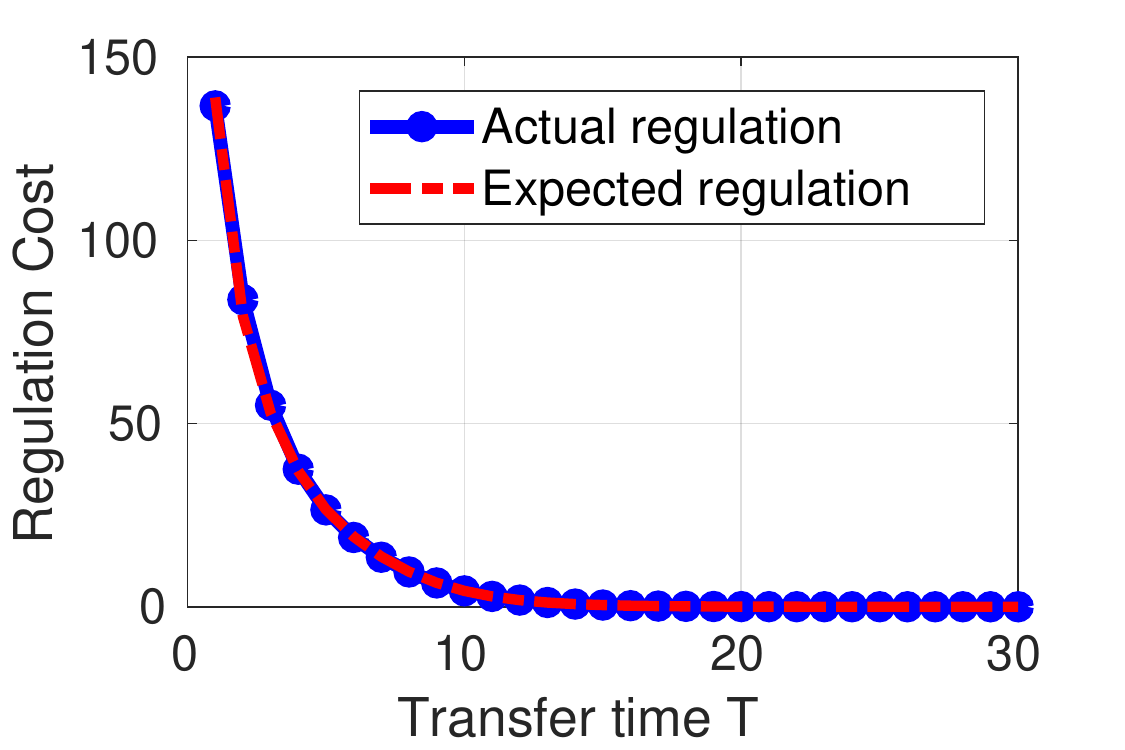}}}
    \sbox2{\subfloat[Error at the end of \\ finite horizon transfer (Exp.3)]{\includegraphics[width=0.48\linewidth]{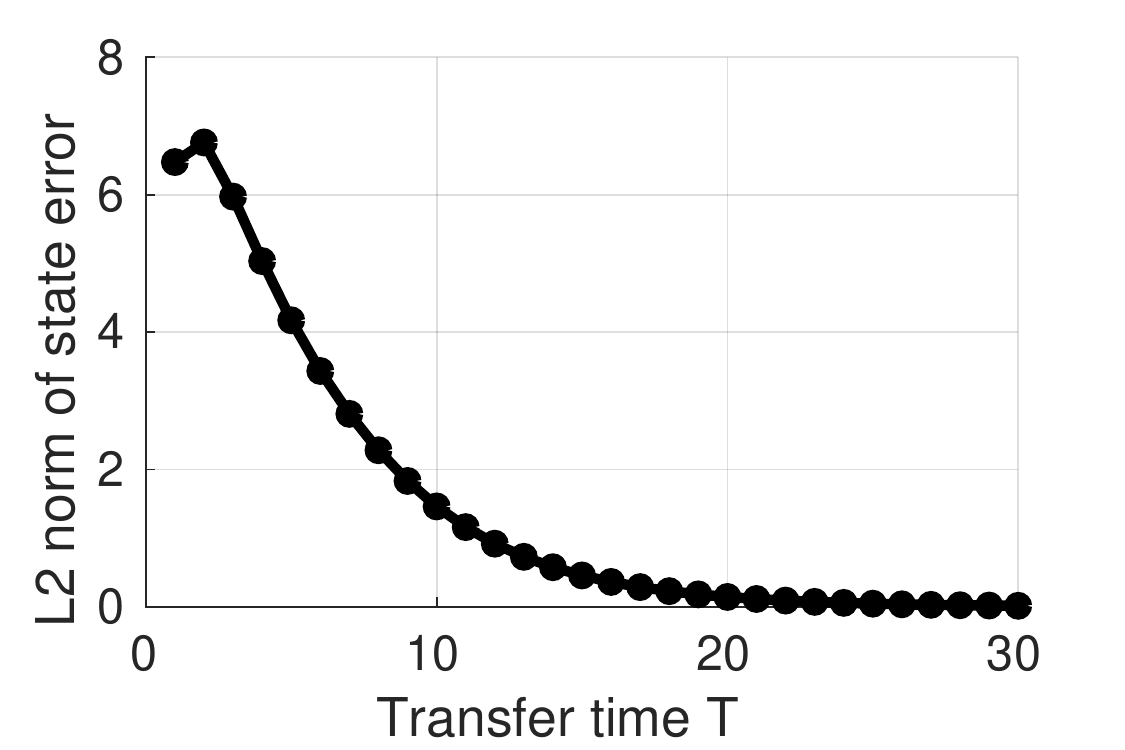}}}
    \sbox3{\subfloat[Total cost (Exp.4)]{\includegraphics[width=0.48\linewidth]{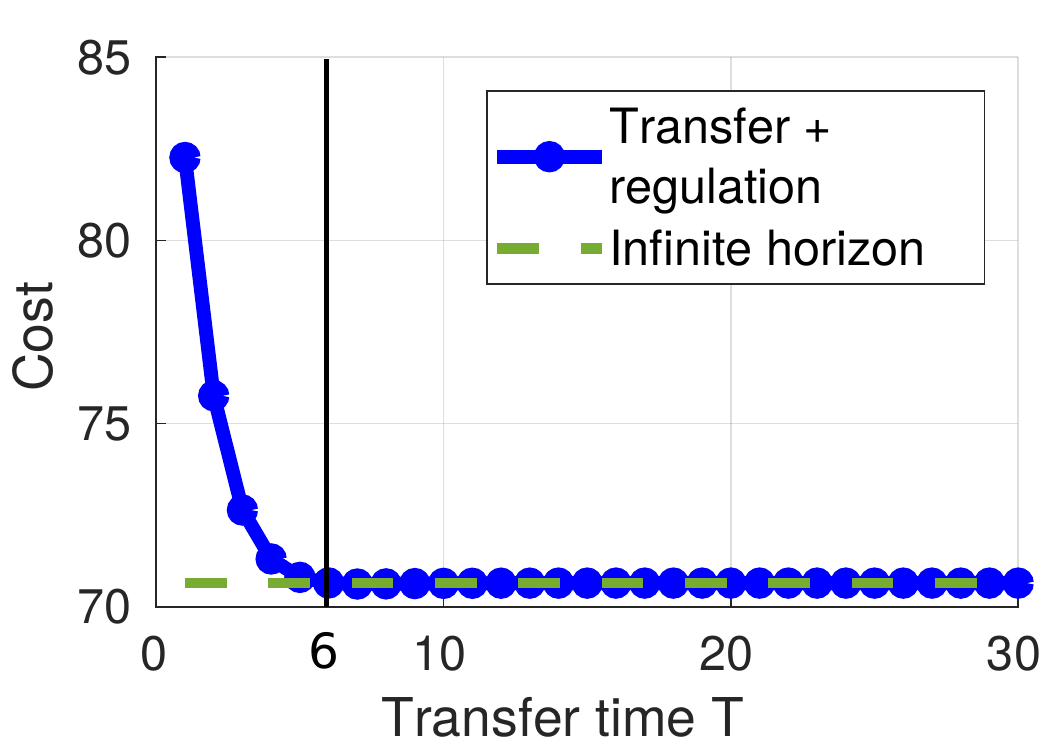}}}
    \sbox4{\subfloat[Terminal regulation cost (Exp.4)]{\includegraphics[width=0.48\linewidth]{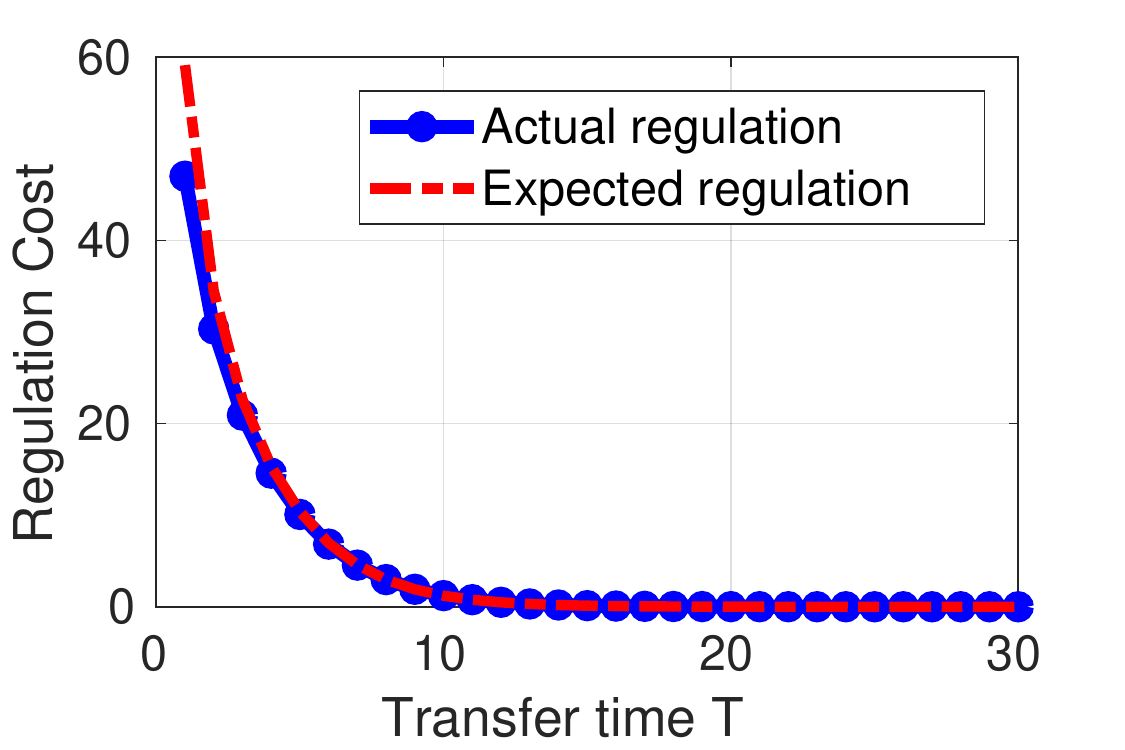}}}
    \sbox5{\subfloat[Error at the end of \\ finite horizon transfer (Exp.4)]{\includegraphics[width=0.48\linewidth]{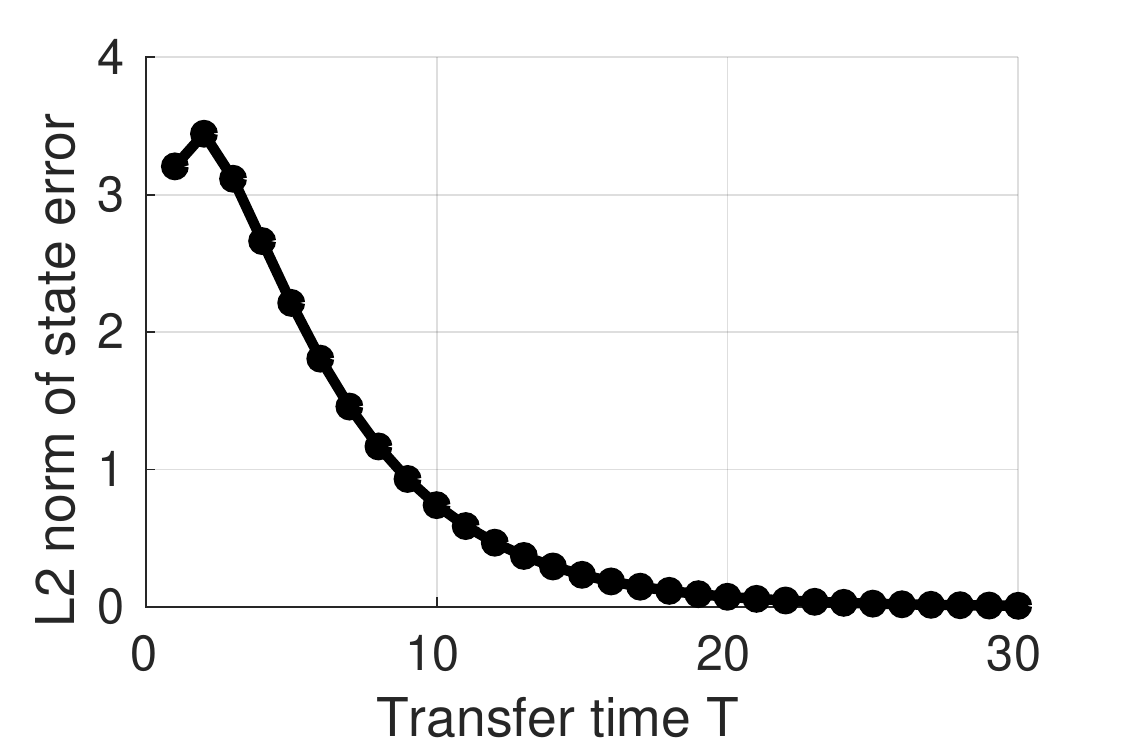}}}
    \centering
    \usebox0\hfil \usebox3\par
  \usebox1\hfil \usebox4\par
  \usebox2\hfil \usebox5
    \caption{Pendulum  swing-up. Note plots labeled (a) - (c) are Exp.3 and (d) - (f) are Exp.4 in Table~\ref{tab:experiments}.}
    \label{fig:pendulum}
\end{figure}

\begin{figure}[!htbp]
    \centering
    \includegraphics[width=\linewidth]{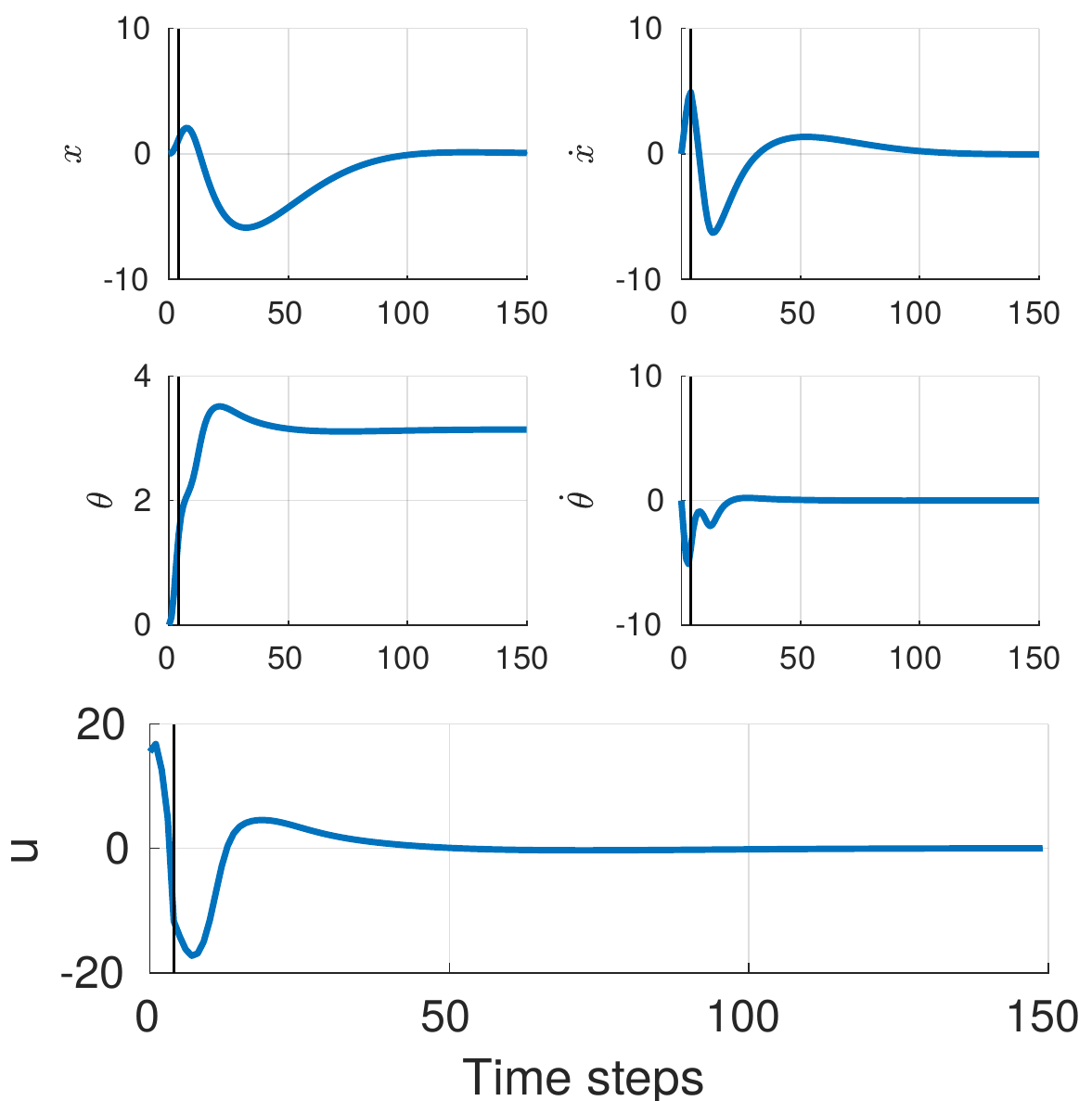}
    \caption{Response of the cart-pole for horizon $T=4$. The black vertical line indicates the change from the finite horizon to the regulation control. Note the jump in the states and control during the transition. The control jumps from 4.8 to -11.7 during the transition from the finite horizon controller to the regulator. }
    \label{fig:cartpole_response}
\end{figure}
\begin{figure}[!htbp]
    \centering
    \includegraphics[width=\linewidth]{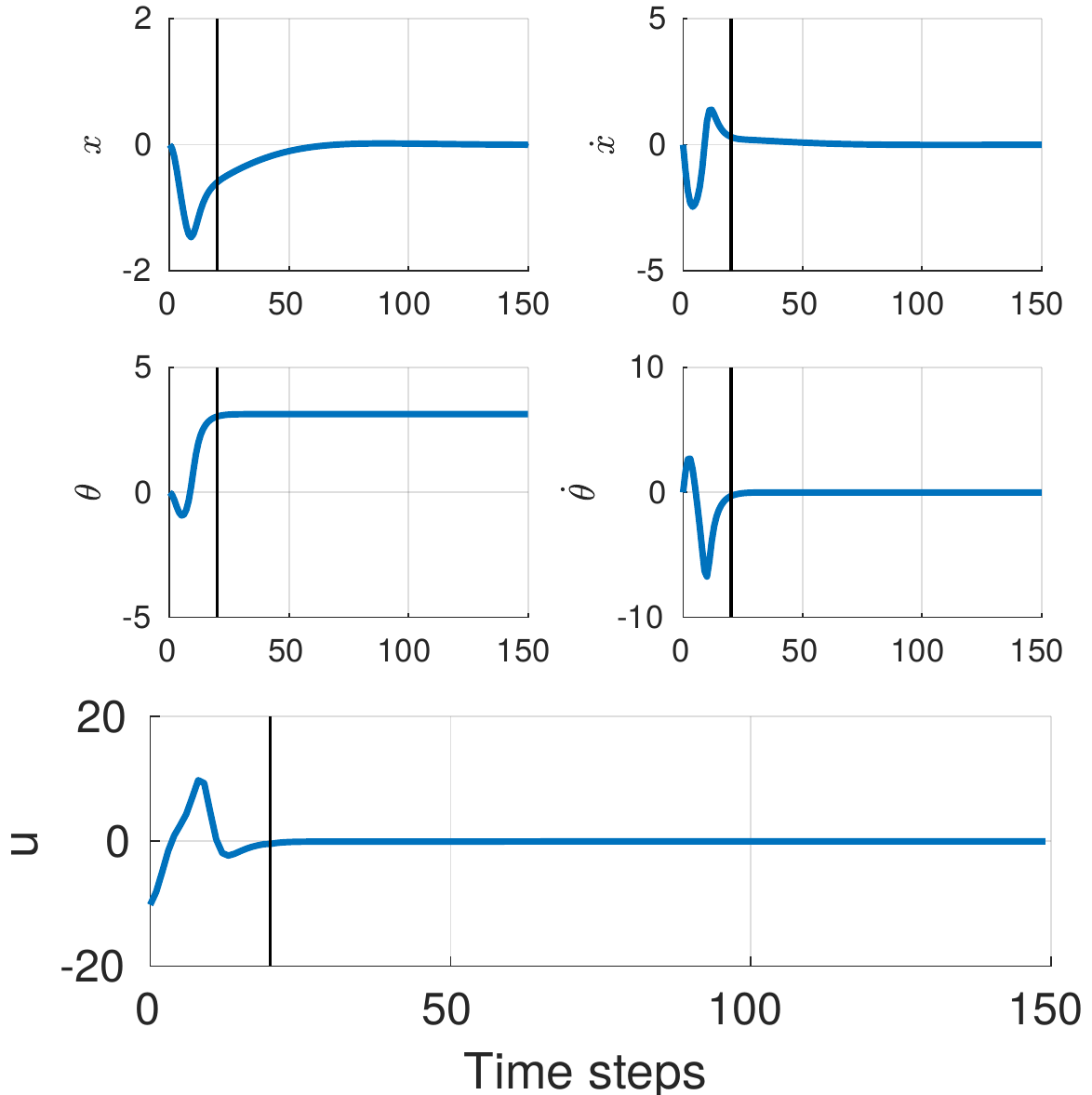}
    \caption{Response of the cart-pole for horizon $T=20$. The black vertical line indicates the change from the finite horizon to the regulation control. The transition is smooth in this case as opposed to \cref{fig:cartpole_response}}
    \label{fig:cartpole_response1}
\end{figure}

Our objective in these experiments is to show that the cost in our formulation of the FH-OCP converges to the infinite horizon optimal cost: $ J^T_\infty(x) = \min_{u_t}  \sum_{t=0}^{T-1} c(x_t,u_t) +\bar{J}_\infty(x_T),$ subject to the nonlinear dynamics, converges to the the optimal cost $J^*_\infty(x)$ as $T$ is increased, i.e. For $T > T^*$, $|J^T_\infty(x) - J^*_\infty| < \varepsilon$. The experiments will also show that the time $T^*$ also depends on the initial condition of the system.
    
To find $\bar{J}_\infty(x)$, we linearize the system around the origin, and since the cost is quadratic, we can calculate $P_\infty$ at the origin using the stationary Riccati equation. Hence, $\bar{J}_\infty(x) = x' P_\infty x$. Using $\bar{J}_\infty(x)$ as the terminal cost for the FH-OCP, we solve it for different values of $T$ using the iterative Linear Quadratic Regulator (iLQR) algorithm for optimization \cite{ILQG_tassa2012synthesis}. It has been shown that one can converge to the unique global optimum for smooth nonlinear problems with control affine dynamics and quadratic control cost by satisfying the necessary conditions for optimal control \cite{mohamed2022acc}, which iLQR does. So, the problem of local optimum is avoided in our experiments. After the finite horizon or `transfer time' $T$, we switch the control to the LQR regulator designed at the origin till $t=150$. Since the true infinite horizon problem cannot be solved, we solve the FH-OCP for $T=150$ without any terminal cost, as a surrogate for the infinite horizon cost which is labeled as ``Infinite horizon" in the plots. We note that solving for smaller horizons does not stabilize the systems to the upright position.

In our plots, the cost mentioned as ``Transfer + regulation" is the total cost of the finite horizon controller and the regulator. In all four experiments, we will observe that the total cost converges after some time $T = T^*$. It can also be seen that the converged cost matches the infinite horizon cost. The experiments also show that the value $T^*$ is dependent on the initial conditions chosen. The expected regulation cost is calculated using $x_T' P_\infty x_T$, where $x_T$ is the terminal state of the FH-OCP. The actual regulation is calculated from the trajectory of the terminal regulator. The mismatch in expected and actual regulation cost implies that the linear terminal regulator is not optimal for the state $x_T$. 

In practice, one would typically solve the infinite horizon OCP by formulating it as a finite horizon OCP with a very large $T$, which would be computationally very expensive. In this case, it is $T= 150$. But by regularizing the finite horizon OCP using the terminal cost $\bar{J}_\infty(\cdot)$ as defined in section \ref{section:sol_IHOCP}, one can solve it for a much smaller horizon $T^*$. Note that in our experiments, $T^* < 15 << 150 = T$ for both the cases considered, and we achieve near-optimality as shown in our results. Solving for a smaller horizon is computationally much cheaper, and also opens up avenues where one has to replan frequently in a stochastic setting. Moreover, since the cost-to-go from our formulation is a control Lyapunov function, it guarantees global asymptotic stability for the origin.

\section{Conclusion}\label{section:Conclusion}
In this paper, we have developed a tractable approach to the approximate solution of nonlinear infinite horizon optimal control problems that is globally asymptotically stabilizing and converges to the true optimal solution in the limit of a vanishing terminal set. Empirical results show that the practical convergence occurs in a very short time compared to the effective horizon required for a solution of the infinite horizon without the free final time and terminal set regularization employed by the approach. Future work will involve the incorporation of state and control constraints and the testing of the approach on a suite of nonlinear problems with varying degrees of complexity which will require a data-based generalization leveraging our prior work \cite{D2C2.0_CDC,arxiv_D2C2.0}. We shall also consider the extension of the approach to the problem of optimal nonlinear output feedback control along with a suitable data-based generalization.
\printbibliography
\end{document}